\documentclass[a4paper, 11pt, reqno]{article}
\usepackage{amsmath}
\usepackage{amssymb,esint}
\usepackage{amscd}
\usepackage{xspace}
\usepackage{fancyhdr}
\usepackage{color}
\usepackage{srcltx} 

\newtheorem{theorem}{Theorem}[section]

\newtheorem{definition}[theorem]{Definition}

\newtheorem{lemma}[theorem]{Lemma}

\newtheorem{remark}[theorem]{Remark}

\newenvironment{proof}[1][Proof]{\textbf{#1.} }{\hfill\rule{0.5em}{0.5em}}
{\catcode`\@=11\global\let\AddToReset=\@addtoreset
\AddToReset{equation}{section}

\AddToReset{theorem}{section}

\newcommand{\norm}[1]{\left \Vert #1\right\Vert}
           
\newcommand{\R}{\mathbb{R}}

\begin{document}
\title{Regularity of solutions for a class of quasilinear elliptic equations related to Caffarelli-Kohn-Nirenberg inequality}
	\author{
	{\bf Le Cong Nhan,\thanks{E-mail address: nhanlc@hcmute.edu.vn, Ho Chi Minh City University of Technology and Education, Ho Chi Minh City, Vietnam.} 		
	~ Ky Ho, \thanks{E-mail address: kyhn@ueh.edu.vn, University of Economics Ho Chi Minh City, Ho Chi Minh City, Vietnam.}	
	Le Xuan Truong\thanks{E-mail address: lxuantruong@ueh.edu.vn, University of Economics Ho Chi Minh City, Ho Chi Minh City, Vietnam.}
	}}
\date{}  
\maketitle
\begin{abstract}

This paper is concerned with a class of quasilinear elliptic equations involving some potentials related to the Caffarelli-Korn-Nirenberg inequality. We prove the local boundedness and H\"older continuity of weak solutions by using the classical De Giorgi techniques. Our result extends the results of Serrin \cite{Serrin64} and Corolado and Peral \cite{CP04}.
 
\medskip

\medskip

\noindent MSC2010: 35B45; 35B65

\medskip

\noindent Keywords: quasilinear equation; Caffarelli-Korn-Nirenberg inequality; H\"older continuity; De Giorgi method
\end{abstract}   
                  
\tableofcontents
									
 \section{Introduction} 
 \qquad Let $\Omega \subset \R^N$ ($N \geq 2$) be a smooth bounded domain with $0 \in \Omega$. The main goal of this work is to prove the H\"older continuity of weak solutions to the following  Dirichlet problem for quasilinear elliptic equations 
 \begin{equation}\label{eq1.1}
 \begin{cases}
 -{\rm div}(\mathcal{A}(x, \nabla u)) =  \mathcal{B}\left(x,u,\nabla u\right), & x \in \Omega, \\
 u = 0, & x \in \partial\Omega,
 \end{cases}
 \end{equation} 							
 where $\mathcal{A} : \Omega \times \R^N \to \R^N$ and  $\mathcal{B} : \Omega \times \R \times \R^N \to \R$ are Carath\'eodory functions satisfying the following growth conditions: 
 \begin{equation}\label{A1}
 \mathcal{A}\left(x,\xi\right)\cdot\xi \geq \lambda \left|x\right|^{-p\gamma} |\xi|^{p}, \quad |\mathcal{A}(x,\xi)|\le \Lambda \left|x\right|^{-p\gamma} |\xi|^{p-1},
 \end{equation}
 and 
 \begin{equation}\label{B1}
 |\mathcal{B}(x, u, \xi)| \leq a(x)\left|x\right|^{-q_1\gamma }\left|\xi\right|^{q_1} + b(x)\left|x\right|^{-q_0\mu }\left|u\right|^{q_0-1} +c(x),
 \end{equation}
 for a.e. $x \in \Omega$ and for all $\xi \in \R^N$. Here $\lambda, \Lambda$ are positive constants and
 \begin{itemize}
 	\item the  constants $\gamma, \mu$ and $p, q_0, q_1$ satisfy following conditions 
 	\begin{equation}\label{C2}
 	- 1 < \gamma < \frac{N-p}{p}; \quad \gamma \leq \mu < 1 + \gamma, 
 	\end{equation}
 	\begin{equation}\label{C1}
 	1<p<N,\quad p\leq q_0<p^*_{\gamma,\mu},\quad q_1 = \frac{p}{q_0'}, 
 	\end{equation}
 	where $q_0':=\frac{q_0}{q_0-1}$ 
 	and $ \displaystyle p^*_{\gamma,\mu} := \frac{Np}{N - (\gamma + 1 - \mu)p}$;
 	
 	\item $a, b, c : \Omega \to \R$ are measurable functions satisfying 
 	\begin{align}
 	&a\in L^s\left(w_a,\Omega\right),\quad w_a(x) = \left|x\right|^{\mu s}, \qquad \qquad \qquad   s>\frac{q_0p_{\gamma,\mu}^*}{p_{\gamma,\mu}^*-q_0},\label{D1}\\
 	& b \in L^t\left(w_b,\Omega\right), \quad w_b(x) = \left|x\right|^{\mu \left({p_{\gamma,\mu}^*(t-1)}-q_0t\right) }, \quad t>\frac{p_{\gamma,\mu}^*}{p_{\gamma,\mu}^*-q_0},\label{D2}\\
 	&c\in L^z\left(w_c,\Omega\right), \quad w_c (x)= \left|x\right|^{\mu p_{\gamma,\mu}^*\left(z-1\right)}, \qquad \qquad \,\, z>\frac{p_{\gamma,\mu}^*}{p_{\gamma,\mu}^*-1}.\label{D3}
 	\end{align}
 	Here and in what follows, the weighted space $L^\varrho(\omega, \Omega)$ is defined as \begin{equation*}
 	L^\varrho(\omega, \Omega) = \left\{f : \Omega \to \R \,\, \text{measurable such that} \,\,  \int_\Omega |f(x)|^\varrho \omega(x)dx < \infty\right\}.
 	\end{equation*}
 	Moreover, we also assume that
 	\begin{align}\label{D4}
 	\delta := \max\left\{\frac{1}{t}+\frac{q_0}{p_{\gamma,\mu}^*},\frac{1}{z}+\frac{1}{p_{\gamma,\mu}^*}\right\} <\left\{
 	\begin{array}{ll}
 	\displaystyle\frac{\left(\gamma+1-\mu\right)p}{N}&\text{if}\quad 0\leq\mu< \gamma+1,\\
 	\\
 	\displaystyle\frac{\left(\gamma+1\right)p}{N}&\text{if}\quad \gamma\leq \mu< 0.\\
 	\end{array}
 	\right.
 	\end{align}
 \end{itemize}
 
 It is well-known that the H\"older continuity of weak solutions to elliptic equations in divergence form with discontinuous coefficients was first proved by De Giorgi  \cite{DeGi57} and, independently, by Nash \cite{Nash58}. Another seminal contribution was made by Moser \cite{Moser60}, who found a new proof of the De Giorgi-Nash theorem by means of the Harnack inequality. These methods are now known as De Giorgi-Nash-Moser techniques. For more detail, we refer the reader to \cite{GM12,GT77,HL97,KS80} for linear elliptic equations and \cite{LU73,Serrin64} for quasilinear elliptic equations involving $p$-Laplacian.
 
 \medskip
Regarding the H\"older regularity for classical quasilinear elliptic equations without weight ($\gamma=\mu=0$ and $q_0=p$), the problem \eqref{eq1.1} was treated by Serrin \cite{Serrin64}. The author used the iteration technique introduced by Moser \cite{Moser60,Moser61} to establish the Harnack inequality and prove the continuity of weak solution to \eqref{eq1.1}. See also Trudinger \cite{Trudinger67} and Ladyzhenskaya and Ural'seva \cite{LU73} for some other generalized quasilinear elliptic equations.
 
 \medskip
 In case the degenerate elliptic equations, Fabes et al. \cite{FKS82} established the local H\"older continuity of weak solutions of certain classes of degenerate elliptic equations $Lu = 0$ with $Lu= \mathrm{div}\left(A(x)\nabla u\right)$ and
 \begin{align*}
 \frac{1}{C}w(x)\left|\xi\right|^2\leq A(x)\xi\cdot\xi\leq Cw(x)\left|\xi\right|^2,
 \end{align*}	
 and the weight $w$ belongs to the Muckenhoupt class $A_2$, that is, for all balls $B\subset \mathbb{R}^N$
 \begin{align*}
 \left(\frac{1}{\left|B\right|}\int_{B}wdx\right)\left(\frac{1}{\left|B\right|}\int_{B}w^{-1}dx\right)\leq C.
 \end{align*}
 This condition guarantees the weighted Sobolev embedding theorems for the weight $w$. A special case of such weight $w(x)=\left|x\right|^\beta$ with $-N<\beta\leq -\left(N-2\right)$ was pointed out by the authors in \cite[Section 3]{FKS82}. In the paper, the authors used the Moser iteration technique to prove the local Harnack inequality and H\"older continuity (\cite[Lemma 2.3.5 and Theorem 2.3.12]{FKS82}). We also refer to \cite{EP72,Kruzkov63,MS68} for the other degenerate elliptic equations.
 
 \medskip
 In particular, Corolado and Peral \cite{CP04} used the De Giorgi technique to prove the H\"older continuity for the solutions to the degenerate elliptic equations
 \begin{align*}
 -\mathrm{div}\left(\left|x\right|^{-\gamma p}\left|\nabla u\right|^{p-2}\nabla u\right) = f\quad\text{in}\,\,\,\Omega,
 \end{align*}
 with mixed Dirichlet-Neumann boundary conditions, where $\Omega\subset\mathbb{R}^N$ is a smooth domain with $0\in \Omega$, and $f\in L^r\left(\left|x\right|^{-\eta r},\Omega\right)$  with $\eta=-p^*\gamma\left(r-1\right)/r$. This is a special case of our setting when $a=b=0$ and $\mu=\gamma$. It is also noticed that although our boundary condition is Dirichlet type, by using similar arguments as in our proof it is not difficult to see that it still holds true with mixed Dirichlet-Neumann boundary conditions. 
 
 \medskip	
 Inspired by these works and by the spirit of the De Giorgi method \cite{DeGi57} we prove the H\"older continuity of weak solutions to the problem \eqref{eq1.1} which generalizes \cite{CP04,FKS82,Serrin64}. The main ingredient to study the equations \eqref{eq1.1} is the Caffarelli-Kohn-Nirengberg inequality \cite{CKN84} (CKN inequality, for short). However even with this useful tool in hand, the main difficulty is due to the functions $a,b,c$ belonging to the weighted Lebesgue spaces related to $\mu$ with $\gamma\leq \mu<\gamma+1$. In other words, we have to deal with two different weights (may either vanish or be infinite or both) appearing in the Lebesgue spaces and Sobolev spaces. Hence it requires more delicate techniques to treat the problem even only for proving the local boundedness of weak solutions (see Section \ref{Sec3}). Also note that our proof is different from \cite{FKS82,Serrin64} where the authors used the Moser iteration technique.

 \medskip
 To state our main theorem we need to define the concept of weak solution for problem \eqref{eq1.1}. For this purpose we introduce the weighted Lebesgue spaces and weighted Sobolev spaces which are defined by
 \begin{equation*}
 L^\varrho(\omega, \Omega) = \left\{f : \Omega \to \R \,\, \text{measurable such that} \,\,  \int_\Omega |f(x)|^\varrho \omega(x)dx < \infty\right\},
 \end{equation*}
 and
 \begin{equation*}
 W^{1,\varrho}(\omega, \Omega) := \left\{f \in L^\varrho(\omega, \Omega) : \, \frac{\partial f}{\partial x_i} \in L^\varrho(\omega, \Omega) \,\,\, \text{for $i = 1, 2, ..., N$} \right\},
 \end{equation*}
 where the derivatives are understood in the sense of distributions. The norm of these spaces are given, respectively, by
 \[
 \|f\|^\varrho_{L^\varrho(\omega, \Omega) } := \int_\Omega |f(x)|^\varrho \omega(x)dx,
 \]
 and
 \[
 \|f\|_{\varrho,\omega} := \left(\|f\|^\varrho_{L^\varrho(\omega, \Omega) } + \|\nabla f\|^\varrho_{L^\varrho(\omega, \Omega) } \right)^{\frac{1}{\varrho}}.
 \]
 
 On the other hand, we also define the Sobolev spaces $W_0^{1,\varrho}(\omega, \Omega)$ by the completion of the space $C_0^\infty(\Omega)$ under the norm $\| \cdot \|_{\varrho, \omega}$. It is noted that by the Poincar\'e inequality, we can use $\|\nabla f\|_{L^\varrho(\omega, \Omega) } $ as an equivalent norm in the space $W_0^{1,\varrho}(\omega, \Omega)$.
 
 \begin{definition}[Weak solutions]
 	A function $u\in W_{0}^{1,p}\left(\omega_0, \Omega\right)$ is called to be a weak solution of problem \eqref{eq1.1} if there holds
 	\begin{align}
 	\int_\Omega \mathcal{A}\left(x,\nabla u\right)\cdot\nabla\varphi dx = \int_\Omega \mathcal{B}\left(x,u,\nabla u\right)\varphi dx, \quad\forall\varphi\in W_{0}^{1,p}\left(\omega_0, \Omega\right).
 	\end{align}
 	Here and in the following, $\omega_0$ 
 	always stands for the weight $|x|^{-\gamma p}$.	
 \end{definition}
 
 Our main result is the following theorem.
 \begin{theorem}\label{thm.01}
 	Assume that the assumptions \eqref{A1}-\eqref{D4} hold true. Then every weak solution $u$ of problem \eqref{eq1.1} is globally H\"older continuous on $\bar{\Omega}$, that is, 
 	\begin{equation*}
 	u \in C^\lambda(\bar{\Omega}), \quad \text{for some $\lambda \in (0, 1)$}.
 	\end{equation*} 
 \end{theorem}

\begin{remark}
	It is noticed that in order to prove the theorem it suffices to show the H\"older continuity of solutions in a neighborhood of the origin. This is because for the other cases and on the boundary	the weight $\left|x\right|^{\alpha}$ is regular and therefore \eqref{eq1.1} reduces to the classical quasilinear elliptic equation with $\gamma=\mu=0$. The proof then can be followed step by step the arguments as in \cite{CP04,Giusti03,KS80}.
\end{remark} 

\noindent\textbf{Notations:}
\begin{itemize}
	\item[$\diamond$]  We denote by $B_r$ the ball center at $0$ with radius $r>0$, and define
	\[
	A(k, r) := \left\{u > k \right\} \cap B_r(0), \quad \text{with} \,\,  k\in\mathbb{R};
	\]
	\item[$\diamond$] If $f$ is a function defined on $\R^N$ then $f_+ = \max\{f, 0\}$;
	\item[$\diamond$] For a measurable subset $E \subset \R^N$ we put $$|E|_\sigma = \int_E |x|^\sigma dx;$$ 
	If $\sigma = 0$ then we shall write $|E|$ instead of $|E|_0$.
	\item[$\diamond$] For a ball $B_r \subset \Omega$ and $u\in L^\infty\left(B_r\right)$, we put
	\[
	M(r) := \sup_{B_{r}}u(x), \quad m(r) := \inf_{B_{r}}u(x). 
	\]
	Here and in what follows, $\sup$ and $\inf$ are understood as $\text{ess}\sup$ and $\text{ess}\inf$, respectively. We also define the oscillation of $u$ on $B_r$ by
	\[
	\text{osc}_{B_r}u := \sup_{B_{r}}u - \inf_{B_{r}}u.
	\]
	\item[$\diamond$] We shall write $A \lesssim B$ whenever $A \leq CB$ for some constant $C > 0$. Moreover, in order to avoid complications of the notation we use the same symbol $C$ to indicate different constants, possibly varying from one passage to the next one.
\end{itemize}

\section{Preliminaries}\label{sec_2}
\qquad For the sake of completeness we collect here some useful inequalities for our goals. We begin by the Caffarelli-Kohn-Nirenberg inequality \cite{CKN84}. 

\begin{lemma}[Caffarelli-Kohn-Nirenberg inequality]\label{lem_CKN}
	Assume that $N\geq 2$. Let $p$, $q$, $\mu$ and $\gamma$ be real constants such that $p, q \geq 1$ and
	\begin{align*}
	\frac{1}{p}-\frac{\gamma}{N}>0,\,\,\text{and}\,\,\,\frac{1}{q}-\frac{\mu}{N}>0.
	\end{align*}
	Then there exists a positive constant $C$ such that for all
	$u \in C_0^\infty\left(\mathbb{R}^N\right)$ we have	
	\begin{align}
	\norm{\left|x\right|^{-\mu} u}_{L^q}\leq C\norm{\left|x\right|^{-\gamma} \left|\nabla u\right|}_{L^p}
	\end{align}
	if and only if 
	\begin{align*}
	\frac{1}{q} - \frac{\mu}{N} = \frac{1}{p} - \frac{\gamma+1}{N}\quad\text{and}\quad 0\leq \mu-\gamma<1.
	\end{align*}	
\end{lemma}

We also need the following lemma which is a variation of the CKN inequality in $W^{1,1}\left(\left|x\right|^{-\gamma},B_R\right)$.
\begin{lemma}\label{lem_CKN2}
	Let $N\geq 2$ and $\alpha$, $q$, $\mu$ and $\gamma$ be real constants such that 
		\begin{align*}
			\alpha>0,\quad q \geq 1, \quad  0 < \frac{1}{q}-\frac{\mu}{N} = 1 - \frac{\gamma+1}{N}, \quad \text{and} \quad 0 \leq  \mu - \gamma < 1.
		\end{align*}
Then there exists a positive constant $C$ such that the inequality
		\begin{align}\label{eq2.01}
			\left(\int_{B_R}\left| |x|^{-\mu} u\right|^qdx \right)^{\frac{1}{q}} \leq C \int_{B_R}\left| |x|^{-\gamma} \nabla u\right| dx,
		\end{align}
	holds for any function $u \in W^{1,1}(|x|^{-\gamma},B_R)$ satisfying the condition $$|\left\{x \in B_R : u = 0\right\}|\geq\alpha.$$	
\end{lemma}
\begin{proof}
Without loss of generality, we assume that $B_R$ is centered at the origin. Fix $\varepsilon\in (0,R)$ such that
\begin{equation}\label{varep}
|B_\varepsilon|<\frac{\alpha}{2}.
\end{equation}
Let  $u \in W^{1,1}(|x|^{-\gamma},B_R)$ with $|\left\{x \in B_R : u = 0\right\}|\geq\alpha.$ By \eqref{varep}, we have
	\begin{align}\label{eq2.02}
	|\left\{x \in B_R \setminus B_\varepsilon: u(x) = 0\right\}|>\frac{\alpha}{2}.
	\end{align}
From this and the fact that $u\in W^{1,1}(B_R \setminus B_\varepsilon)$, in view of \cite[Theorem 3.16]{Giusti03} we have
	\begin{align}\label{E}
	\left(\int_{B_R \setminus B_\varepsilon}\left|u\right|^{1^*}dx \right)^{\frac{1}{1^*}} \leq C(N,R,\varepsilon,\alpha) \int_{B_R \setminus B_\varepsilon}\left| \nabla u\right| dx,
	\end{align} 
where $1^*:=\frac{N}{N-1}$.	Next, we decompose
	\begin{equation*}
		\int_{B_R}\left| |x|^{-\mu} u\right|^qdx = \int_{B_\varepsilon}\left| |x|^{-\mu} u\right|^qdx  + \int_{B_R \setminus B_\varepsilon}\left| |x|^{-\mu} u\right|^qdx.
	\end{equation*}
	From $q=\frac{N}{N-\left(\gamma+1-\mu\right)}\leq 1^*$, using the H\"older inequality and \eqref{E} we obtain
	\begin{align}\label{CKN.1}
	\notag	& \int_{B_R \setminus B_\varepsilon}\left| |x|^{-\mu} u\right|^qdx  \leq \max\left\{ \varepsilon^{-\mu q}, R^{-\mu q}\right\}\int_{B_R \setminus B_\varepsilon}\left| u\right|^qdx \\
	\notag	& \qquad \leq \max\left\{ \varepsilon^{-\mu q}, R^{-\mu q}\right\} \left|B_R \setminus B_\varepsilon \right|^{1 - \frac{q}{1^*}}\left(\int_{B_R \setminus B_\varepsilon}\left| u\right|^{1^*} dx \right)^\frac{q}{1^*} \\
		\notag	& \qquad \leq C(N,R,\varepsilon,\alpha,\mu,q)\left(\int_{B_R \setminus B_\varepsilon}\left| \nabla u\right|dx \right)^q\\
		\notag	& \qquad \leq C(N,R,\varepsilon,\alpha,\mu,q)\max\left\{ \varepsilon^{\gamma q}, R^{\gamma q}\right\}\left(\int_{B_R \setminus B_\varepsilon}\left| |x|^{-\gamma} \nabla u\right|dx \right)^q\\
		& \qquad \leq C(N,R,\varepsilon,\alpha,\mu,q,\gamma)\left(\int_{B_R \setminus B_\varepsilon}\left| |x|^{-\gamma} \nabla u\right|dx \right)^q.
	\end{align}
Let $\eta \in C_0^\infty(B_R)$ be such that
	\[
		0 \leq \eta \leq 1, \quad \eta \equiv 1 \,\, \text{on} \,\, B_\varepsilon.
	\]
	Since $\eta \equiv 1 \,\, \text{on} \,\, B_\varepsilon$ and $\text{supp} (\eta u) \subset B_{R}$ we imply
	\begin{align*}
	& \int_{B_\varepsilon}\left| |x|^{-\mu} u\right|^qdx \leq \int_{B_{R}}\left| |x|^{-\mu} (\eta u)\right|^qdx \leq C\left(\int_{B_{R}}\left| |x|^{-\gamma} |\nabla(\eta u)|\right| dx \right)^q
	\end{align*}
	by using the Caffarelli-Kohn-Nirenberg inequality, Lemma \ref{lem_CKN}. Hence it follows that
	\begin{align*}
		\int_{B_\varepsilon}\left| |x|^{-\mu} u\right|^qdx& \leq C\left(\int_{B_R\setminus B_\varepsilon}\left| |x|^{-\gamma} |\nabla\eta| |u|\right| dx + \int_{B_{R}}\left| |x|^{-\gamma} \eta |\nabla u|\right| dx\right)^q\\
		&  \leq C\left(\|\nabla\eta\|_\infty \max\left\{\varepsilon^{-\gamma}, R^{-\gamma} \right\}\int_{B_R\setminus B_\varepsilon}\left|u\right| dx + \int_{B_{R}}\left| |x|^{-\gamma} |\nabla u|\right| dx\right)^q.
	\end{align*}
	Applying the H\"older inequality and then using \eqref{E}, we easily deduce from the last inequality that 
		\begin{align}\label{CKN.2}
	\int_{B_\varepsilon}\left| |x|^{-\mu} u\right|^qdx \leq C(N,R,\varepsilon,\alpha,\gamma,\eta)\left(\int_{B_{R}}\left| |x|^{-\gamma} |\nabla u|\right| dx\right)^q.
	\end{align}
	By combining \eqref{CKN.1} with \eqref{CKN.2} we obtain \eqref{eq2.01}. 
	The proof of our lemma is completed.
\end{proof}

\vspace{0.2cm}
Next we state here two following classical results which its proof can be found in \cite{Giusti03}.

\begin{lemma}\cite[Lemma 7.1, page 220]{Giusti03}\label{lem_iteration1}
	Let $\{U_n\}_{n=1}^\infty$ be a sequence of positive numbers such that
		\begin{align*}
			U_{n+1}\leq CB^nU_n^{1+\alpha},
		\end{align*}
	where $B >1$ and $C, \alpha > 0$. Then  $U_n\to 0$	as $n\to \infty$ provided 
		\begin{align*}
			U_0\leq C^{-\frac{1}{\alpha}}B^{-\frac{1}{\alpha^2}}.
		\end{align*}
\end{lemma}

\begin{lemma}\cite[Lemma 7.3, page 229]{Giusti03} \label{lem_iteration2}
	Let $\varphi : \R_+ \to \R_+$ be a non-decreasing function. Suppose that there exists a constant $\tau \in (0, 1)$ such that for every $0 < R < R_0$ we have
		\begin{align*}
			\varphi\left(\tau R\right) \leq \tau^\alpha\varphi(R) + BR^\beta
		\end{align*}
	with $0 < \beta < \alpha$ and $B \geq 0$. Then there exists a constant $C = C\left(\tau,\alpha,\beta\right)$ such that 
		\begin{align*}
			\varphi\left(r\right)\leq C\left[\left(\frac{r}{R}\right)^\alpha\varphi(R) + B r^\beta\right],
		\end{align*}
		for all $0 < r < R \leq R_0$.
\end{lemma}

\medskip
We end this section by proving the following Caccioppoli type inequality on level sets for problem \eqref{eq1.1}. 

\begin{lemma}\label{theo_Caccio}
	Suppose that $u\in W_{0}^{1,p}\left(\omega_0, \Omega\right)$ is a weak solution of \eqref{eq1.1}.  Then there exist positive constants 
	\[
		R_0 = R_0\Big(a, \norm{\left|x\right|^{-\mu}u}_{L^{p_{\gamma,\mu}^*}}\Big) \quad \text{and} \quad C = C\Big(\lambda, \Lambda, b, c, \norm{\left|x\right|^{-\mu}u}_{L^{p_{\gamma,\mu}^*}}\Big)
	\]
	such that for every $k \geq 0$, and $0 < r < R <R_0$ we have
		\begin{align}\label{eq3.1}
		\int_{A(k,r)}\left|x\right|^{-p\gamma}\left|\nabla u\right|^pdx \leq \frac{C}{\left(R-r\right)^p}\int_{A(k,R)}\left|x\right|^{-p\gamma}(u-k)^pdx+ C\left|A\left(k,R\right)\right|_{-\mu p_{\gamma,\mu}^*}^{1-\delta}.
		\end{align}
	where $\delta\in\left(0,1\right)$ is a constant defined by \eqref{D4}.	Similarly \eqref{eq3.1} still holds for $-u$.
\end{lemma}	
\begin{proof}
	Let $0<r<R$ and $\eta\in C_0^\infty\left(B_R\right)$  be a cut-off function satisfying
		\begin{align*}
			0\leq \eta\leq 1, \quad\eta=1 \,\,\text{on}\,\,B_r,\quad \left|\nabla \eta\right|\lesssim \frac{1}{R-r}.
		\end{align*}
	For $k\geq 0$, let $v=\left(u-k\right)_+\eta^p$ then $v\in W_{0}^{1,p}\left(\omega_0, \Omega\right)$ and
		\begin{align*}
			\nabla v=\eta^p\nabla (u-k)_++p(u-k)_+\eta^{p-1}\nabla \eta.
		\end{align*}
	By testing the equation \eqref{eq1.1} by $v$, we find that
		\begin{align}\label{eq2.1}
			\int_\Omega \eta^p\mathcal{A}\left(x,\nabla u\right)\cdot\nabla(u-k)_+dx =& -p\int_\Omega (u-k)_+\eta^{p-1}A\left(x,\nabla u\right)\cdot\nabla \eta dx\notag\\
			&+\int_\Omega \mathcal{B}\left(x,u,\nabla u\right)(u-k)_+\eta^pdx.
		\end{align}
By the assumption \eqref{A1}  we can estimate the left hand side and the first term in the right hand side of \eqref{eq2.1} as follows
	\begin{align}\label{eq2.2}
		\int_\Omega \eta^p\mathcal{A}\left(x,\nabla u\right)\cdot\nabla(u-k)_+dx \geq\lambda \int_{A(k,R)}\left|x\right|^{-p\gamma}\eta^p\left|\nabla u\right|^pdx,
	\end{align}
and
	\begin{align}
		&-p\int_\Omega (u-k)_+\eta^{p-1}\mathcal{A}\left(x,\nabla u\right)\cdot\nabla\eta dx\notag\\
		&\qquad\leq \Lambda p\int_{A(k,R)}\left|x\right|^{-p\gamma}\eta^{p-1}\left|\nabla u\right|^{p-1}\left|(u-k)_+\nabla \eta\right|dx\notag\\
		&\qquad\leq \frac{\lambda}{4}\int_{A(k,R)}\left|x\right|^{-p\gamma}\eta^p\left|\nabla u\right|^pdx + C(\lambda,\Lambda)\int_{A(k,R)}\left|x\right|^{-p\gamma}\left|(u-k)_+\nabla\eta\right|^pdx.\label{eq2.3}
	\end{align}
By assumption \eqref{B1} we have that
	\begin{align}\label{eq2.4}
		\int_\Omega \mathcal{B}\left(x,u,\nabla u\right)(u-k)_+\eta^pdx\leq& \int_{A(k,R)}a\left|x\right|^{-q_1\gamma }\left|\nabla u\right|^{q_1}(u-k)_+\eta^pdx\notag\\ 
		&+\int_{A(k,R)} b\left|x\right|^{-q_0\mu }\left|u\right|^{q_0-1}(u-k)_+\eta^pdx\notag\\ &+\int_{A(k,R)}c(u-k)_+\eta^pdx.
	\end{align}
We now estimate each terms in the right hand side of \eqref{eq2.4}. Since $q_1 = {p}/{q'_0}$, by using the Young inequality the first term in the right hand side of \eqref{eq2.4} yields
	\begin{align}\label{eq2.5}
		&\int_{A(k,R)}a\left|x\right|^{-q_1\gamma }\left|\nabla u\right|^{q_1}(u-k)_+\eta^pdx \notag\\
		&\qquad\qquad\leq \frac{\lambda}{4}\int_{A(k,R)}\left|x\right|^{-p\gamma}\eta^p\left|\nabla u\right|^pdx + C(\lambda) \int_{A(k,R)}\eta^p\left[a\left(u-k\right)_+\right]^{q_0}dx.
	\end{align}
By virtue of the H\"older and Caffarelli-Kohn-Nirenberg inequality we have 
	\begin{align}\label{eq2.6}
		&\int_{A(k,R)}\eta^p\left[a\left(u-k\right)_+\right]^{q_0}dx\notag\\
		&\qquad = \int_{A(k,R)}\left(\left|x\right|^\mu a\right)^{q_0}\left(\left|x\right|^{-\mu}(u-k)_+\right)^{{q_0}-p}\left(\left|x\right|^{-\mu}(u-k)_+\eta\right)^pdx\notag\\
		&\qquad\leq \epsilon(R)\left[\int_{A(k,R)}\left(\left|x\right|^{-\mu}(u-k)_+\eta\right)^{p_{\gamma,\mu}^*}dx\right]^\frac{p}{p_{\gamma,\mu}^*}\notag\\
		&\qquad\leq \epsilon(R)\int_{A(k,R)}\left|\left|x\right|^{-\gamma}\nabla\left[(u-k)_+\eta\right]\right|^{p}dx\notag\\
		&\qquad\leq 2^p\epsilon(R)\int_{A(k,R)}\left|x\right|^{-p\gamma}\eta^p\left|\nabla u\right|^pdx + 2^p\epsilon(R)\int_{A(k,R)}\left|x\right|^{-p\gamma}\left|(u-k)_+\nabla\eta\right|^pdx,
	\end{align}
where $\epsilon(R)$ is given by
	\begin{align*}
		\epsilon(R) = \left\{\int_{A(k,R)}\left[\left(\left|x\right|^\mu a\right)^{q_0}\left(\left|x\right|^{-\mu}(u-k)_+\right)^{{q_0}-p}\right]^\frac{p_{\gamma,\mu}^*}{p_{\gamma,\mu}^*-p}dx\right\}^{1-\frac{p}{p_{\gamma,\mu}^*}}.
	\end{align*}	
Using \eqref{D1} and the H\"older inequality we have
	\begin{align}\label{eq2.7}
		\epsilon(R) \leq \norm{a}_{L^s(\omega_a, \Omega)}^{q_0}\norm{\left|x\right|^{-\mu}u}_{L^{p_{\gamma,\mu}^*}}^{{q_0}-p}\left|A(k,R)\right|^{\varepsilon_1}
	\end{align}	
where $\varepsilon_1\in (0,1)$ holds
	\begin{align*}
		\frac{{q_0}}{s} + \frac{{q_0}-p}{p_{\gamma,\mu}^*} +\varepsilon_1=1-\frac{p}{p_{\gamma,\mu}^*}.
	\end{align*}
Choose $R_0$ small enough (depending on $\norm{a}_{L^s(\omega_a, \Omega)}$ and $\norm{\left|x\right|^{-\mu}u}_{L^{p_{\gamma,\mu}^*}}$) to be such that 
	\begin{align*}
		2^p\epsilon(R_0)C(\lambda)\leq \frac{\lambda}{4}
	\end{align*}
then we deduce from \eqref{eq2.5}-\eqref{eq2.7} that for any $R < R_0$,
	\begin{align}\label{eq2.8}
	&\int_{A(k,R)}a\left|x\right|^{-q_1\gamma }\left|\nabla u\right|^{q_1}(u-k)_+\eta^pdx \notag\\
	&\qquad\qquad\leq \frac{\lambda}{2}\int_{A(k,R)}\left|x\right|^{-p\gamma}\eta^p\left|\nabla u\right|^pdx + \frac{\lambda}{4} \int_{A(k,R)}\left|x\right|^{-p\gamma}\left|(u-k)_+\nabla\eta\right|^pdx.
	\end{align}
For the second term in the right hand side of \eqref{eq2.4}. Since $t >\frac{p_{\gamma,\mu}^*}{p_{\gamma,\mu}^*-q_0}$ we can find $\varepsilon_2\in (0,1)$ to be such that
	\begin{align*}
	\frac{1}{t}+\frac{q_0}{p_{\gamma,\mu}^*} =1-\varepsilon_2.
	\end{align*}
Setting $d\nu=\left|x\right|^{-\mu p_{\gamma,\mu}^*}dx$, by using \eqref{D2} and the H\"older's inequality we find that
	\begin{align}\label{eq2.9}
		&\int_{A(k,R)} b\left|x\right|^{-q_0\mu }\left|u\right|^{q_0-1}(u-k)_+\eta^pdx \notag\\
		&\qquad\leq \int_{A(k,R)} \left|x\right|^{\mu\left(p_{\gamma,\mu}^*-q_0\right) }b\left|u\right|^{q_0}d\nu\notag\\
		&\qquad\leq \norm{b}_{L^t\left(w_b,\Omega\right)}\norm{\left|x\right|^{-\mu}u}_{L^{p_{\gamma,\mu}^*}}^{q_0}\left|A\left(k,R\right)\right|_{-\mu p_{\gamma,\mu}^*}^{\varepsilon_2}.
	\end{align}	
For the last term in \eqref{eq2.4}. It follows from \eqref{D3} that there exists $\varepsilon_3\in\left(0,1\right)$ such that
	\begin{align*}
	\frac{1}{z}+\frac{1}{p_{\gamma,\mu}^*}=1-\varepsilon_3.
	\end{align*}
Applying the H\"older inequality we have
	\begin{align}\label{eq2.10}
		&\int_{A(k,R)}c(u-k)_+\eta^pdx \notag\\
		&\qquad\leq  \int_{A(k,R)}\left(\left|x\right|^{\frac{\mu p_{\gamma,\mu}^*\left(z-1\right)}{z}}c\right)\left(\left|x\right|^{-\mu}(u-k)_+\right)\left|x\right|^{-\mu\frac{p_{\gamma,\mu}^*(z-1)-z}{z}}dx\notag\\
		&\qquad\leq \norm{c}_{L^z(\omega_c, \Omega)}\norm{\left|x\right|^{-\mu}(u-k)_+}_{L^{p_{\gamma,\mu}^*}}\left|A\left(k,R\right)\right|_{-\mu p_{\gamma,\mu}^*}^{\varepsilon_3}.
	\end{align}
Choosing $\delta=\max\left\{\frac{1}{t}+\frac{q_0}{p_{\gamma,\mu}^*},\frac{1}{z}+\frac{1}{p_{\gamma,\mu}^*}\right\}$ and combining \eqref{eq2.8}--\eqref{eq2.10}, then \eqref{eq2.4} yields
	\begin{align}\label{eq2.11}
		\int_\Omega \mathcal{B}\left(x,u,\nabla u\right)&(u-k)_+\eta^pdx\leq \frac{\lambda}{2}\int_{A(k,R)}\left|x\right|^{-p\gamma}\eta^p\left|\nabla u\right|^pdx \notag\\
		&+ \frac{\lambda}{4} \int_{A(k,R)}\left|x\right|^{-p\gamma}\left|(u-k)_+\nabla\eta\right|^pdx+ C\left|A\left(k,R\right)\right|_{-\mu p_{\gamma,\mu}^*}^{1-\delta},
	\end{align}
for some constant $C$ depending on $b,c$ and $\norm{\left|x\right|^{-\mu}u}_{L^{p_{\gamma,\mu}^*}}$.

\medskip
From \eqref{eq2.1}, \eqref{eq2.2}, \eqref{eq2.3} and \eqref{eq2.11}, we arrive at
	\begin{align*}
		\int_{A(k,R)}\left|x\right|^{-p\gamma}\eta^p\left|\nabla u\right|^pdx \leq C\int_{A(k,R)}\left|x\right|^{-p\gamma}\left|(u-k)_+\nabla\eta\right|^pdx+ C\left|A\left(k,R\right)\right|_{-\mu p_{\gamma,\mu}^*}^{1-\delta} 
	\end{align*}
and hence the proof follows by the definition of $\eta$.
\end{proof}	

\section{Local boundedness}\label{Sec3}
In this section, we proceed the first step in the De Giorgi method, that is, the local boundedness of weak solutions to \eqref{eq1.1}.
\begin{theorem}\label{thrm-localbound}
	Let $u\in W_{0}^{1,p}\left(w_0,\Omega\right)$ be a weak solution to \eqref{eq1.1} satisfying \eqref{eq3.1}. Then $u$ is locally bounded from above in $\Omega$.
\end{theorem}
\begin{proof}
	Fix $0<r<\rho<R<R_0\leq 1$, let $\eta\in C_0^\infty\left(B_\rho\right)$ such that 
	\begin{align*}
	0\leq \eta\leq 1, \quad\eta=1 \,\,\text{on}\,\,B_r,\quad \left|\nabla \eta\right|\lesssim \frac{1}{\rho-r}.
	\end{align*}
	By Lemma \ref{theo_Caccio} we have that
	\begin{align}
	&\int_{B_1}\left|x\right|^{-p\gamma}\left|\nabla(u-k)_+\eta\right|^pdx \notag\\
	&\qquad\lesssim \int_{B_1}\left|x\right|^{-p\gamma}(u-k)_+^p\left|\nabla\eta\right|^pdx + \int_{B_1}\left|x\right|^{-p\gamma}\eta^p\left|\nabla(u-k)_+\right|^pdx\notag\\
	&\qquad\lesssim\frac{1}{\left(\rho-r\right)^p}\int_{A\left(k,\rho\right)}\left|x\right|^{-p\gamma}(u-k)_+^pdx \notag\\
	&\qquad\quad+ \frac{1}{\left(R-\rho\right)^p}\int_{A(k,R)}\left|x\right|^{-p\gamma}\left|(u-k)_+\right|^pdx+ \left|A\left(k,R\right)\right|_{-\mu p_{\gamma,\mu}^*}^{1-\delta}. \label{eq4.2}
	\end{align}	
	Choosing $\rho=\frac{R+r}{2}$ then \eqref{eq4.2} yields
	\begin{align}
	&\int_{B_1}\left|x\right|^{-p\gamma}\left|\nabla(u-k)_+\eta\right|^pdx\notag\\
	&\qquad\qquad\lesssim \frac{1}{\left(R-r\right)^p}\int_{A(k,R)}\left|x\right|^{-p\gamma}(u-k)^pdx+ \left|A\left(k,R\right)\right|_{-\mu p_{\gamma,\mu}^*}^{1-\delta}.\label{eq4.3}
	\end{align}
	We now let
		\begin{equation*}
		\mathcal{I}_{\mu}(k, r) := \begin{cases}
		\displaystyle \int_{A(k,r)} |x|^{-\mu p^*_{\mu,\gamma}}(u - k)^{p^*_{\mu,\gamma}}dx, & \text{if $0\leq \mu <\gamma+1$}, \\
		\\
		\displaystyle \int_{A(k,r)} (u - k)^pdx, & \text{if $\gamma\leq\mu < 0$}. 
		\end{cases}
		\end{equation*}
	and consider separately two cases of $\mu$, namely, $\gamma\leq\mu< 0$ and $0\leq \mu<\gamma+1$.
	
	\medskip
	{\bf Case 1: $0\leq \mu <\gamma+1$.} 
	By virtue of the Caffarelli-Kohn-Nirenberg inequality, we deduce from \eqref{eq4.3} that
	\begin{align}
	\mathcal{I}_{\mu}^\frac{p}{p_{\gamma,\mu}^*}\left(k,r\right) \lesssim& \frac{1}{\left(R-r\right)^p}\int_{A(k,R)}\left|x\right|^{-p\gamma}(u-k)^pdx+ \left|A\left(k,R\right)\right|_{-\mu p_{\gamma,\mu}^*}^{1-\delta}\notag\\
	\lesssim &\frac{R^{(\mu-\gamma) p}}{\left(R-r\right)^p}\int_{A(k,R)}\left|x\right|^{-\mu p}(u-k)^pdx+ \left|A\left(k,R\right)\right|_{-\mu p_{\gamma,\mu}^*}^{1-\delta}\notag\\
	\lesssim & \frac{1}{\left(R-r\right)^p}\left|A(k,R)\right|^{1-\frac{p}{p_{\gamma,\mu}^*}}\mathcal{I}_{\mu}^\frac{p}{p_{\gamma,\mu}^*}\left(k,R\right) + \left|A\left(k,R\right)\right|_{-\mu p_{\gamma,\mu}^*}^{1-\delta},\label{eq4.5}
	\end{align}
	Since $\mu\geq 0$ and $R<1$, we have $\mu p_{\gamma,\mu}^*-\gamma p\geq 0$ and
	\begin{align*}
	\left|A\left(k,R\right)\right|\leq R^{\mu p_{\gamma,\mu}^*}\left|A\left(k,R\right)\right|_{-\mu p_{\gamma,\mu}^*}\leq \left|A\left(k,R\right)\right|_{-\mu p_{\gamma,\mu}^*},
	\end{align*}	
	and hence, \eqref{eq4.5} yields
	\begin{align}
	\mathcal{I}_\mu\left(k,r\right)\lesssim& \frac{1}{\left(R-r\right)^{p_{\gamma,\mu}^*}}\left|A(k,R)\right|_{-\mu p_{\gamma,\mu}^*}^{\frac{p_{\gamma,\mu}^*-p}{p}}\mathcal{I}_\mu\left(k,r\right) + \left|A\left(k,R\right)\right|_{-\mu p_{\gamma,\mu}^*}^{\frac{\left(1-\delta\right)p_{\gamma,\mu}^*}{p}}\notag\\
	\lesssim&\frac{1}{\left(R-r\right)^{p_{\gamma,\mu}^*}}\mathcal{I}_\mu\left(k,r\right)\left|A(k,R)\right|_{-\mu p_{\gamma,\mu}^*}^{\epsilon_1} + \left|A\left(k,R\right)\right|_{-\mu p_{\gamma,\mu}^*}^{1+\epsilon_1},\label{eq4.6}
	\end{align}
	where $\epsilon_1$ is given by
	\begin{align*}
	\epsilon_1=\frac{\left(p_{\gamma,\mu}^*-p\right)-\delta p_{\gamma,\mu}^* }{p}=\frac{\left(\gamma+1-\mu\right)p-\delta N}{p}>0.
	\end{align*}
	
	\medskip
	Since $A(k,R)\subset A(h,R)$ for $0<h<k$, we have that
	\begin{align}\label{eq4.7}
	\mathcal{I}_{\mu}(k,R)\leq \int_{A(k,R)}\left|x\right|^{-\mu p_{\gamma,\mu}^*}(u-h)^{p_{\gamma,\mu}^*}dx \leq \mathcal{I}_{\mu}(h,R).
	\end{align}
	Since $u-h>k-h$ on $A(k,R)$, we find that
	\begin{align}\label{eq4.8}
	\left|A(k,R)\right|_{-\mu p_{\gamma,\mu}^*}\leq& \frac{1}{(k-h)^{p_{\gamma,\mu}^*}}\int_{A(k,R)}\left|x\right|^{-\mu p_{\gamma,\mu}^*}\left(u-h\right)^{p_{\gamma,\mu}^*}dx \leq \frac{1}{(k-h)^{p_{\gamma,\mu}^*}}\mathcal{I}_{\mu}(h,R).
	\end{align}
	Combining \eqref{eq4.6}-\eqref{eq4.8} we obtain
	\begin{align*}
	\mathcal{I}_{\mu}(k,r)\lesssim \frac{1}{(k-h)^{\epsilon_1p_{\gamma,\mu}^*}}\left[\frac{1}{\left(R-r\right)^{p_{\gamma,\mu}^*}}+ \frac{1}{(k-h)^{p_{\gamma,\mu}^*}}\right]\mathcal{I}_{\mu}^{1+\epsilon_1}(h,R).
	\end{align*}
	Hence, since $\epsilon_1>0$ the standard iteration gives us the $L^\infty$-bound for $u$.
	
	\medskip
	{\bf Case 2: $\gamma\leq \mu< 0$.} Using again the H\"older inequality and \eqref{eq4.3} we have that
	\begin{align}
	\mathcal{I}_{\mu}\left(k,r\right) \leq& \left|A(k,r)\right|^{1-\frac{p}{p_{\gamma}^*}}\left(\int_{A(k,r)}\left[(u-k)_+\eta\right]^{p_{\gamma}^*}dx\right)^\frac{p}{p_{\gamma}^*}\notag\\
	\lesssim & \left|A(k,R)\right|^{\frac{(\gamma+1)p}{N}}\int_{B_1}\left|x\right|^{-p\gamma}\left|\nabla(u-k)_+\eta\right|^pdx
	\notag\\
	\lesssim&\left|A(k,R)\right|^{\frac{(\gamma+1)p}{N}}\left(\frac{1}{\left(R-r\right)^p}\int_{A(k,R)}\left|x\right|^{-p\gamma}(u-k)^pdx +\left|A\left(k,R\right)\right|_{-\mu p_{\gamma,\mu}^*}^{1-\delta}\right)\notag\\
	\lesssim&\left|A(k,R)\right|^{\frac{(\gamma+1)p}{N}}\left(\frac{R^{-\gamma p}}{\left(R-r\right)^p}\mathcal{I}_\mu\left(k,R\right) +R^{-\mu p_{\gamma,\mu}^*(1-\delta)}\left|A\left(k,R\right)\right|^{1-\delta}\right).\label{eq3.8}
	\end{align}
	Since $A(k,R)\subset A(h,R)$ for $0<h<k$, we have that
	\begin{align}\label{eq3.9}
	\mathcal{I}_\mu(k,R)\leq \int_{A(k,R)}(u-h)^{p}dx \leq \mathcal{I}_\mu(h,R).
	\end{align}
	Again, from $u-h>k-h$ on $A(k,R)$ we find that
	\begin{align}\label{eq3.10}
	\left|A(k,R)\right|\leq& \frac{1}{(k-h)^{p}}\int_{A(k,R)}\left(u-h\right)^{p}dx \leq  \frac{1}{(k-h)^{p}}\mathcal{I}_\mu(h,R).
	\end{align}
	Since $\gamma\leq \mu<0$ we deduce from \eqref{eq3.8}-\eqref{eq3.10} that
	\begin{align*}
	\mathcal{I}_\mu\left(k,r\right)\lesssim&\frac{1}{\left(R-r\right)^p}\mathcal{I}_\mu\left(k,R\right)\left|A(k,R)\right|^{\frac{p(\gamma+1)}{N}} +\left|A\left(k,R\right)\right|^{1+\epsilon_2}\\
	\lesssim& \frac{1}{\left(k-h\right)^{\epsilon_2p}}\left[\frac{1}{\left(R-r\right)^p}+\frac{1}{\left(k-h\right)^p}\right]\mathcal{I}_\mu^{1+\epsilon_2}(h,R).
	\end{align*}	
	where $\epsilon_2=\frac{(\gamma+1)p-\delta N}{N}>0$ due to \eqref{D4}. Thus by standard iteration we imply that $u$ is locally bounded. The proof is complete.
\end{proof}

\section{Local H\"older continuity}
First of all, it is worth noticing that by using the local boundedness (Theorem \ref{thrm-localbound}) and following the proof of Theorem \ref{theo_Caccio}, we find that
\begin{align}\label{theo_Caccio2}
	\int_{A(k,r)}\left|x\right|^{-p\gamma}\left|\nabla u\right|^pdx \leq \frac{C_{\textit{data}}}{\left(R-r\right)^p}\int_{A(k,R)}\left|x\right|^{-p\gamma}(u-k)^pdx+ C_{\textit{data}}\left|A\left(k,R\right)\right|_{-\mu p_{\gamma,\mu}^*}^{1-\delta},
\end{align}
for every $k \in \mathbb{R}$ with $\left|k\right|+\sup_{B_{R_0}}\left|u\right|\leq M$, and $0 < r < R <R_0$,  $\delta\in\left(0,1\right)$ as in \eqref{D4} and $C_{\textit{data}}$ depends only to $\lambda,\Lambda,a,b,c,M$ and $N$.

\subsection{Quantitative local bounds}
We give in this section two lemmas (Lemma \ref{thm3.1} and Lemma \ref{lmz02}) analogous but more quantitative than Theorem \ref{thrm-localbound}. For its proof we denote as previous section
	\begin{equation*}
	\mathcal{I}_{\mu}(k, r) := \begin{cases}
	\displaystyle \int_{A(k,r)} |x|^{-\mu p^*_{\mu,\gamma}}(u - k)^{p^*_{\mu,\gamma}}dx, & \text{if $0\leq \mu <\gamma+1$}, \\
	\\
	\displaystyle \int_{A(k,r)} (u - k)^pdx, & \text{if $\gamma\leq\mu < 0$}. 
	\end{cases}
	\end{equation*}

\begin{lemma}[The case $0 \leq \mu < \gamma + 1$] \label{thm3.1}
	We put
	\begin{align*}
		\xi_1 
		 = \left(\frac{N-p}{p}-\gamma\right)\frac{(\gamma+1-\mu)p-\delta N}{N-(\gamma+1-\mu)p},\quad\text{and}\quad \epsilon_1 = \frac{(\gamma+1-\mu)p-\delta N}{N-(\gamma+1-\mu)p}>0,
	\end{align*}
	and take $\alpha_1 > 0$	to be such that $\alpha_1\left(1+\alpha_1\right)=\epsilon_1$.  
	
	\vspace{0.2cm}\noindent	
	Then, for any $k_0 \in \R$ with $\left|k_0\right|+\sup\left|u\right|\leq M$, the following estimate holds true:
	\begin{align*}
	&\sup\limits_{B_{R/2}}u(x)\leq  k_0+R^{\xi_1} \\
	&\quad + 2^{\frac{1+\alpha_1}{\alpha_1^{2}}}C_{\textit{data}}^{\frac{1}{\alpha_1p_{\gamma,\mu}^*}}\left(\frac{\left|A(k_0,R)\right|_{-\mu p_{\gamma,\mu}^*}}{R^{N-\mu p_{\gamma,\mu}^*}}\right)^\frac{\alpha_1}{p_{\gamma,\mu}^*}\left(\frac{1}{R^{N-\mu p_{\gamma,\mu}^*}}\int_{A(k_0,R)}\left|x\right|^{-\mu p_{\gamma,\mu}^*}\left(u-k_0\right)^{p_{\gamma,\mu}^*}dx\right)^\frac{1}{p_{\gamma,\mu}^*},
	\end{align*}
where $C_{data}$ is as in \eqref{theo_Caccio2}.
\end{lemma}
\begin{proof} Let $0 < r < \rho < R$ and $\eta\in C_0^\infty\left(B_\rho\right)$ be a test function satisfying
	\begin{align*}
	0\leq \eta\leq 1, \quad\eta=1 \,\,\text{on}\,\,B_r,\quad \left|\nabla \eta\right|\lesssim \frac{1}{\rho-r}.
	\end{align*}
	Then using similar arguments as in the proof of Theorem \ref{thrm-localbound} in the case $0\leq \mu<\gamma+1$, we obtain
	\begin{align}
		\mathcal{I}_{\mu}^\frac{p}{p_{\gamma,\mu}^*}\left(k,r\right) \leq& \frac{C_{\textit{data}}}{\left(R-r\right)^p}\int_{A(k,R)}\left|x\right|^{-p\gamma}(u-k)^pdx+ C_{\textit{data}}\left|A\left(k,R\right)\right|_{-\mu p_{\gamma,\mu}^*}^{1-\delta}\notag\\
	\lesssim &\frac{R^{(\mu-\gamma) p}}{\left(R-r\right)^p}\int_{A(k,R)}\left|x\right|^{-\mu p}(u-k)^pdx+ \left|A\left(k,R\right)\right|_{-\mu p_{\gamma,\mu}^*}^{1-\delta}\notag\\
	\lesssim & \frac{R^{(\mu-\gamma) p}}{\left(R-r\right)^p}\left|A(k,R)\right|^{1-\frac{p}{p_{\gamma,\mu}^*}}\mathcal{I}_{\mu}^\frac{p}{p_{\gamma,\mu}^*}\left(k,R\right) + \left|A\left(k,R\right)\right|_{-\mu p_{\gamma,\mu}^*}^{1-\delta}.\label{eq3.05}
	\end{align}
	Since $\mu\geq 0$, we have $\mu p_{\gamma,\mu}^*-\gamma p>0$ and
	\begin{align*}
	\left|A\left(k,R\right)\right|\leq R^{\mu p_{\gamma,\mu}^*}\left|A\left(k,R\right)\right|_{-\mu p_{\gamma,\mu}^*}.
	\end{align*}	
	Therefore, it implies from \eqref{eq3.05} that
	\begin{align}\label{eq3.06}
		\mathcal{I}_{\mu}\left(k,r\right) & \lesssim \frac{R^{(\mu-\gamma) p_{\gamma,\mu}^*}}{\left(R-r\right)^{p_{\gamma,\mu}^*}}\left|A(k,R)\right|^{\frac{p_{\gamma,\mu}^*-p}{p}}\mathcal{I}_{\mu}\left(k,R\right) + \left|A\left(k,R\right)\right|_{-\mu p_{\gamma,\mu}^*}^{\frac{\left(1-\delta\right)p_{\gamma,\mu}^*}{p}} \nonumber \\
		& \lesssim \frac{R^{\frac{p^*_{\gamma,\mu}\big(\mu p^*_{\gamma\mu} - \gamma p\big)}{p}}}{\left(R-r\right)^{p_{\gamma,\mu}^*}}\left|A(k,R)\right|_{-\mu p_{\gamma,\mu}^*}^{\frac{p_{\gamma,\mu}^*-p}{p}}\mathcal{I}_{\mu}\left(k,R\right) + \left|A\left(k,R\right)\right|_{-\mu p_{\gamma,\mu}^*}^{{1+\frac{\left(p_{\gamma,\mu}^*-p\right)-\delta p_{\gamma,\mu}^* }{p}}}.
	\end{align}

\vspace{0.2cm}
\quad Now for $k > h$ we have $A(k,R)\subset A(h,R)$ and
\begin{align}\label{eq3.08}
\mathcal{I}_{\mu}(k,R)\leq  \mathcal{I}_{\mu}(h,R).
\end{align}
On the other hand, it finds that, for any $\varrho > 0$,
\begin{align}\label{eq3.09}
\left|A(k,\varrho)\right|_{-\mu p_{\gamma,\mu}^*}\leq& \frac{1}{(k-h)^{p_{\gamma,\mu}^*}}\int_{A(k,\varrho)}\left|x\right|^{-\mu p_{\gamma,\mu}^*}\left(u-h\right)^{p_{\gamma,\mu}^*}dx 
\leq \frac{1}{(k-h)^{p_{\gamma,\mu}^*}}\mathcal{I}_{\mu}(h,\varrho).
\end{align}
We then combine \eqref{eq3.06}, \eqref{eq3.08} and \eqref{eq3.09} to obtain
\begin{align}
\mathcal{I}_{\mu}(k,r)\leq C_{\textit{data}} \left[\frac{R^{\beta}}{\left(R-r\right)^{p_{\gamma,\mu}^*}}+ \frac{1}{(k-h)^{p_{\gamma,\mu}^*}}\right]\mathcal{I}_{\mu}(h,R)\left|A\left(k,R\right)\right|_{-\mu p_{\gamma,\mu}^*}^{\epsilon_1}, \label{eq3.12}
\end{align}
where $\beta$ and $\epsilon_1$ are constants given by
	\begin{align*}
		&\beta = \frac{p^*_{\gamma,\mu}(\mu p^*_{\gamma,\mu} - \gamma p)}{p} - \frac{\delta p^*_{\gamma,\mu}\left(N-\mu p^*_{\gamma,\mu}\right)}{p},\\ &\epsilon_1=\frac{\left(p_{\gamma,\mu}^*-p\right)-\delta p_{\gamma,\mu}^* }{p}=\frac{\left(\gamma+1-\mu\right)p-\delta N}{p}>0.
	\end{align*}
From \eqref{eq3.12} we imply
\begin{align*}
\mathcal{I}_{\mu}(k,r)\leq & \, C_{\textit{data}}\,R^{\beta-p_{\gamma,\mu}^*}\left[\left(\frac{R}{R-r}\right)^{p_{\gamma,\mu}^*}+ \left(\frac{R^{\xi_1}}{k-h}\right)^{p_{\gamma,\mu}^*}\right] \mathcal{I}_{\mu}(h,R)\left|A\left(k,R\right)\right|_{-\mu p_{\gamma,\mu}^*}^{\epsilon_1},
\end{align*}
where $\xi_1$ is a constant given by
	\begin{align*}
		\xi_1=\frac{{p_{\gamma,\mu}^*}-\beta}{{p_{\gamma,\mu}^*}}=\left(\frac{N-p}{p}-\gamma\right)\epsilon_1\quad \text{with} \quad \beta - p_{\gamma,\mu}^*= -\frac{N\left(N-(\gamma+1)p\right)}{N-(\gamma+1-\mu)p}\epsilon_1.
	\end{align*}
Hence we have that
\begin{align*}
\mathcal{I}_{\mu}(k,r)\leq & \, C_{\textit{data}}\,R^{-\frac{N\left(N-(\gamma+1)p\right)}{N-(\gamma+1-\mu)p}\epsilon_1}\left[\left(\frac{R}{R-r}\right)^{p_{\gamma,\mu}^*}+ \left(\frac{R^{\xi_1}}{k-h}\right)^{p_{\gamma,\mu}^*}\right] \mathcal{I}_{\mu}(h,R)\left|A\left(k,R\right)\right|_{-\mu p_{\gamma,\mu}^*}^{\epsilon_1}.
\end{align*}

Consider the function 
\begin{align*}
\phi\left(k,r\right) = \mathcal{I}_{\mu}(k,r)\left|A(k,r)\right|_{-\mu p_{\gamma,\mu}^*}^{\alpha_1},
\end{align*}
with $\alpha_1>0$ to be such that $\alpha_1\left(1+\alpha_1\right)=\epsilon_1$.

\medskip
Then the last inequality yields
\begin{align}
\phi\left(k,r\right)\leq & \, C_{\textit{data}}\,R^{-\frac{N\left(N-(\gamma+1)p\right)}{N-(\gamma+1-\mu)p}\epsilon_1}\left[\left(\frac{R}{R-r}\right)^{p_{\gamma,\mu}^*}+ \left(\frac{R^{\xi_1}}{k-h}\right)^{p_{\gamma,\mu}^*}\right]\frac{\phi^{1+\epsilon_1}\left(h,R\right)}{\left(k-h\right)^{\alpha_1 p_{\gamma,\mu}^*}}.\label{eq3.13}
\end{align}	

Let $k_0 \in \R$ and $d\geq R^{\xi_1}$. We define the sequences
\begin{align*}
k_i = k_0 + d\left(1-2^{-i}\right)\quad\text{and}\quad r_i=\frac{R}{2}\left(1+2^{-i}\right), \quad i \in \{1, 2, ....\}.
\end{align*}
Denote $\phi_i=\phi\left(k_i,r_i\right)$ and apply \eqref{eq3.13} with $k=k_{i+1}$, $h=k_i$, $r=r_{i+1}$ and $R=r_{i}$ we obtain
\begin{align*}
\phi_{i+1}\leq& C_{\textit{data}}\,\left(\frac{2^{i+1}}{d}\right)^{\alpha_1 p_{\gamma,\mu}^*}\left[2^{(i+2) p_{\gamma,\mu}^*} + 2^{(i+1) p_{\gamma,\mu}^*}\left(\frac{R^{\xi_1}}{d}\right)^{p_{\gamma,\mu}^*}\right]R^{-\frac{N\left(N-(\gamma+1)p\right)}{N-(\gamma+1-\mu)p}\epsilon_1}\phi_{i}^{1+\alpha_1}\notag\\
\leq &C_{\textit{data}}\,2^{i(1+\alpha_1)p_{\gamma,\mu}^*}d^{-\alpha_1 p_{\gamma,\mu}^*}R^{-\frac{N\left(N-(\gamma+1)p\right)}{N-(\gamma+1-\mu)p}\epsilon_1}\phi_{i}^{1+\alpha_1}.
\end{align*}
Setting $U_i=d^{-p_{\gamma,\mu}^*}\phi_i$ then the last inequality can be rewritten as follows
\begin{align*}
U_{i+1}\leq C_{\textit{data}}\,2^{i(1+\alpha_1)p_{\gamma,\mu}^*}R^{-\frac{N\left(N-(\gamma+1)p\right)}{N-(\gamma+1-\mu)p}\epsilon_1}U_{i}^{1+\alpha_1}.
\end{align*}	
By virtue of Lemma \ref{lem_iteration1}, $U_i$ converges to $0$	as $i \rightarrow \infty$  if we have 
$$
U_0=d^{-p_{\gamma,\mu}^*}\phi_0\leq 2^{-\frac{(1+\alpha_1) p^*_{\gamma,\mu}}{\alpha_1^{2}}}C_{\textit{data}}^{-\frac{1}{\alpha_1}}\left[R^{-\frac{N\left(N-(\gamma+1)p\right)}{N-(\gamma+1-\mu)p}\epsilon_1}\right]^{-\frac{1}{\alpha_1}},
$$ 
that is
\begin{align*}
d \geq & \, 2^{\frac{1+\alpha_1}{\alpha_1^{2}}}C_{\textit{data}}^{\frac{1}{\alpha_1p_{\gamma,\mu}^*}}R^{-\frac{N-(\gamma+1)p}{\alpha_1 p}\epsilon_1}\phi_0^\frac{1}{p_{\gamma,\mu}^*}\notag\\
=& 2^{\frac{1+\alpha_1}{\alpha_1^{2}}}C_{\textit{data}}^{\frac{1}{\alpha_1p_{\gamma,\mu}^*}}\left(\frac{\left|A(k_0,R)\right|_{-\mu p_{\gamma,\mu}^*}}{R^{N-\mu p_{\gamma,\mu}^*}}\right)^\frac{\alpha_1}{p_{\gamma,\mu}^*}\left(\frac{1}{R^{N-\mu p_{\gamma,\mu}^*}}\int_{A(k_0,R)}\left|x\right|^{-\mu p_{\gamma,\mu}^*}\left(u-k_0\right)^{p_{\gamma,\mu}^*}dx\right)^\frac{1}{p_{\gamma,\mu}^*}.
\end{align*}
Hence if we take 
\begin{align*}
&d=R^{\xi_1} + 2^{\frac{1+\alpha_1}{\alpha_1^{2}}}C_{\textit{data}}^{\frac{1}{\alpha_1p_{\gamma,\mu}^*}}\left(\frac{\left|A(k_0,R)\right|_{-\mu p_{\gamma,\mu}^*}}{R^{N-\mu p_{\gamma,\mu}^*}}\right)^\frac{\alpha_1}{p_{\gamma,\mu}^*}\\
&\qquad\qquad\qquad\times\left(\frac{1}{R^{N-\mu p_{\gamma,\mu}^*}}\int_{A(k_0,R)}\left|x\right|^{-\mu p_{\gamma,\mu}^*}\left(u-k_0\right)^{p_{\gamma,\mu}^*}dx\right)^\frac{1}{p_{\gamma,\mu}^*},
\end{align*}
then $\phi\left(d_0 + d,\frac{R}{2}\right)=0$
\begin{align*}
&\sup\limits_{B_{R/2}}u(x)\leq k_0+R^{\xi_1} \\&\quad+ 2^{\frac{1+\alpha_1}{\alpha_1^{2}}}C_{\textit{data}}^{\frac{1}{\alpha_1p_{\gamma,\mu}^*}}\left(\frac{\left|A(k_0,R)\right|_{-\mu p_{\gamma,\mu}^*}}{R^{N-\mu p_{\gamma,\mu}^*}}\right)^\frac{\alpha_1}{p_{\gamma,\mu}^*}\left(\frac{1}{R^{N-\mu p_{\gamma,\mu}^*}}\int_{A(k_0,R)}\left|x\right|^{-\mu p_{\gamma,\mu}^*}\left(u-k_0\right)^{p_{\gamma,\mu}^*}dx\right)^\frac{1}{p_{\gamma,\mu}^*}.
\end{align*}
This completes the proof.	
\end{proof}

\begin{lemma}[The case $\gamma \leq \mu < 0$] \label{lmz02} We set
\begin{equation*}
	\xi_2 =\frac{\epsilon_2 N -\mu p_{\gamma,\mu}^*(1-\delta)}{p} \quad\text{and}\quad \epsilon_2 = \frac{(\gamma + 1)p-\delta N}{N} > 0,
\end{equation*}
and take $\alpha_2  > 0$	to be such that $\alpha_2 \left(1+\alpha_2\right)=\epsilon_2$.  

\vspace{0.2cm}\noindent
Then, for any $k_0 \in \R$ with $\left|k_0\right|+\sup\left|u\right|\leq M$, there holds 
\begin{align*}
	\sup\limits_{B_{R/2}}u(x)\leq k_0 + R^{\xi_2} + 2^{\frac{1+\alpha_2}{\alpha_2^{2}}}C_{\textit{data}}^{\frac{1}{\alpha_2p}}\left(\frac{\left|A(k_0,R)\right|}{R^{N}}\right)^\frac{\alpha_2}{p}\left(\frac{1}{R^{N}}\int_{A(k_0,R)}\left(u-k_0\right)^{p}dx\right)^\frac{1}{p}.
\end{align*} 
\end{lemma}
\begin{proof}
Let $0 < r < \rho < R$ and $\eta\in C_0^\infty\left(B_\rho\right)$ be a test function satisfying
\begin{align*}
	0\leq \eta\leq 1, \quad\eta=1 \,\,\text{on}\,\,B_r,\quad \left|\nabla \eta\right|\lesssim \frac{1}{\rho-r}.
\end{align*}
Using the same way as in the proof of Theorem \ref{thrm-localbound} in the case $\gamma\leq \mu<0$, we obtain
\begin{align}
\mathcal{I}_\mu\left(k,r\right) \leq& \left|A(k,r)\right|^{1-\frac{p}{p_{\gamma}^*}}\left(\int_{A(k,r)}\left[(u-k)\eta\right]^{p_{\gamma}^*}dx\right)^\frac{p}{p_{\gamma}^*}\notag\\
\leq & \left|A(k,R)\right|^{\frac{p(\gamma+1)}{N}}\int_{B_1}\left|x\right|^{-p\gamma}\left|\nabla((u-k)_+\eta)\right|^pdx
\notag\\
\leq&\left|A(k,R)\right|^{\frac{p(\gamma+1)}{N}}\left[\frac{C_{\textit{data}}}{\left(R-r\right)^p}\int_{A(k,R)}\left|x\right|^{-p\gamma}(u-k)^pdx +C_{\textit{data}}\left|A\left(k,R\right)\right|_{-\mu p_{\gamma,\mu}^*}^{1-\delta}\right]\notag\\
\lesssim&\frac{R^{-\gamma p}}{\left(R-r\right)^p}\left|A(k,R)\right|^{\frac{p(\gamma+1)}{N}}\mathcal{I}_p\left(k,R\right) +R^{-\mu p_{\gamma,\mu}^*(1-\delta)}\left|A\left(k,R\right)\right|^{1+\epsilon_2},\label{eq3.07}
\end{align}
where $\epsilon_2=\frac{(\gamma+1)p-\delta N}{N}>0$ thanks to \eqref{D4}.

\medskip
\quad Now for $h < k$ we have $A(k,R)\subset A(h,R)$ and
\begin{align}\label{eqn3.08}
\mathcal{I}_{\mu}(k,R)\leq  \mathcal{I}_{\mu}(h,R).
\end{align}
Moreover, we also find that
\begin{align}\label{eqn3.11}
\left|A(k,R)\right|\leq& \frac{1}{(k-h)^{p}}\int_{A(k,R)}\left(u-h\right)^{p}dx \leq  \frac{1}{(k-h)^{p}}\mathcal{I}_\mu(h,R).
\end{align}
From \eqref{eq3.07}, \eqref{eqn3.08} and $\eqref{eqn3.11}$ it follows that
\begin{align}
	\mathcal{I}_\mu\left(k,r\right)\leq& C_{\textit{data}}\left[\frac{R^{p-\epsilon_2 N}}{\left(R-r\right)^p}+\frac{R^{-\mu p_{\gamma,\mu}^*(1-\delta)}}{\left(k-h\right)^p}\right]\mathcal{I}_\mu(k,R)\left|A(k,R)\right|^{\epsilon_2} \notag\\
	\leq&C_{\textit{data}}\left[\frac{R^{p}}{\left(R-r\right)^p}+\left(\frac{R^{\xi_2}}{k-h}\right)^p\right]R^{-\epsilon_2 N}\mathcal{I}_\mu(k,R)\left|A(k,R)\right|^{\epsilon_2},\label{eq5.20}
\end{align}	
where $\xi_2$ is a positive constant given by
	\begin{align*}
		\xi_2 = \frac{\epsilon_2 N -\mu p_{\gamma,\mu}^*(1-\delta)}{p}>0.
	\end{align*}
We put
\begin{align*}
\phi\left(k,r\right) = I_{p}(k,r)\left|A(k,r)\right|^{\alpha_2},
\end{align*}
where $\alpha_2>0$ satisfying $\alpha_2\left(1+\alpha_2\right)=\epsilon_2$.

\medskip
Then we deduce from \eqref{eq5.20} that
\begin{align}
\phi\left(k,r\right)\leq & C_{\textit{data}}\left[\left(\frac{R}{R-r}\right)^{p}+ \left(\frac{R^{\xi_2}}{k-h}\right)^{p}\right]\frac{R^{-\epsilon_2 N}}{\left(k-h\right)^{\alpha_2 p}}\phi^{1+\alpha_2}\left(h,R\right).\label{eq3.16}
\end{align}	
Now let $d\geq R^{\xi_2}$	and define
\begin{align*}
k_i= k_0 + d\left(1-2^{-i}\right)\quad\text{and}\quad r_i=\frac{R}{2}\left(1+2^{-i}\right).
\end{align*}
Denote $\phi_i=\phi\left(k_i,r_i\right)$ and apply \eqref{eq3.16} with $k=k_{i+1}$, $h=k_i$, $r=r_{i+1}$ and $R=r_{i}$ we obtain
\begin{align*}
\phi_{i+1}\leq& C_{\textit{data}}\,\left(\frac{2^{i+1}}{d}\right)^{\alpha_2 p}\left[2^{(i + 2)p} + 2^{(i+1) p}\left(\frac{R^{\xi_2}}{d}\right)^{p}\right]R^{-\epsilon_2 N}\phi_{i}^{1+\alpha_2}\\
\leq& C_{\textit{data}}\,2^{i(1+\alpha_2)p}d^{-\alpha_2 p}R^{-\epsilon_2 N}\phi_{i}^{1+\alpha_2}.
\end{align*}
Setting $U_i=d^{-p}\phi_i$ then we can rewrite the last inequality
\begin{align*}
U_{i+1}\leq C_{\textit{data}}2^{i(1+\alpha_2)p}R^{-\epsilon_2 N}U_{i}^{1+\alpha_2}.
\end{align*}	
By Lemma \ref{lem_iteration1}, if $U_0=d^{-p}\phi_0\leq  2^{-\frac{\left(1+\alpha_2\right)p}{\alpha_2^2}}C_{\textit{data}}^{-\frac{1}{\alpha_2}}R^{\frac{N\epsilon_2}{\alpha_2}}$, that is,
\begin{align*}
d\geq& 2^{\frac{1+\alpha_2}{\alpha_2^2}}C_{\textit{data}}^{\frac{1}{\alpha_2p}}R^{-\frac{N\epsilon_2}{p\alpha_2}}\phi_0^\frac{1}{p}\notag\\
=& 2^{\frac{1+\alpha_2}{\alpha_2^2}}C_{\textit{data}}^{\frac{1}{\alpha_2p}}\left(\frac{\left|A(k_0,R)\right|}{R^{N}}\right)^\frac{\alpha_2}{p}\left(\frac{1}{R^{N}}\int_{A(k_0,R)}\left(u-k_0\right)^{p}dx\right)^\frac{1}{p}.
\end{align*}
then $U_i\to 0$ as $i\to \infty$. So if we take 
\begin{align*}
d=R^{\xi_2} + 2^{\frac{1+\alpha_2}{\alpha_2^2}}C_{\textit{data}}^{\frac{1}{\alpha_2p}}\left(\frac{\left|A(k_0,R)\right|}{R^{N}}\right)^\frac{\alpha_2}{p}\left(\frac{1}{R^{N}}\int_{A(k_0,R)}\left(u-k_0\right)^{p}dx\right)^\frac{1}{p},
\end{align*}
then $\phi\left(d,\frac{R}{2}\right)=0$, that is,
\begin{align*}
\sup\limits_{B_{R/2}}u(x)\leq k_0+R^{\xi_2} + 2^{\frac{1+\alpha_2}{\alpha_2^2}}C_{\textit{data}}^{\frac{1}{\alpha_2p}}\left(\frac{\left|A(k_0,R)\right|}{R^{N}}\right)^\frac{\alpha_2}{p}\left(\frac{1}{R^{N}}\int_{A(k_0,R)}\left(u-k_0\right)^{p}dx\right)^\frac{1}{p}.
\end{align*}
This completes the proof.		
\end{proof}

\subsection{Decay of level sets}

\quad In this section we will give the decay of level sets in two cases of $0\leq\mu<\gamma+1$ and $\gamma\leq \mu<0$ separately.

\begin{lemma}[Case $0\leq \mu<\gamma+1$] \label{lem_5.2}
Let 
\[
	k_0 = \frac{M(2R) + m(2R)}{2}.
\]
Assume that weak solution $u$ is bounded and satisfies the condition
	\begin{align}\label{eq5.6}
		\left|A\left(k_0,R\right)\right|_{-\mu p_{\gamma,\mu}^*}\leq \theta_1\left|B_R\left(0\right)\right|_{-\mu p_{\gamma,\mu}^*}
	\end{align}
for some $\theta_1 < 1$. Then we have 
	\begin{align}\label{eq5.8}
		\left(\frac{\left|A(k_\ell,R)\right|_{-\mu p_{\gamma,\mu}^*}}{R^{N-\mu p_{\gamma,\mu}^*}}\right)^{\frac{N-\left(\gamma+1-\mu\right)}{N}\frac{p}{p-1}} \leq &\frac{\theta_1}{\ell}C_{\textit{data}}^\frac{1}{p-1},
	\end{align}	
	for any integer number $\ell$ satisfying 
	\begin{align*}
		\mathrm{osc}_{B_{2R}}u\geq 2^{\ell+1}R^{\xi_1},
	\end{align*}	
	with $\xi_1$ is given by Lemma \ref{thm3.1} and the level $k_\ell$ is defined by 
	$
		k_\ell = M(2R) -\frac{1}{2^{\ell+1}}\mathrm{osc}_{B_{2R}}u.
	$ and $C_{\textit{data}}$ given by \eqref{theo_Caccio2}
\end{lemma}
\begin{proof} 
	We set
	\begin{align*}
		\zeta := \gamma+\frac{\mu p_{\gamma,\mu}^*}{p'}=\gamma +\frac{\mu N(p-1)}{N-\left(\gamma+1-\mu\right)p},\quad\text{and}\quad
		\sigma := \frac{\mu p\left[N-\left(\gamma+1-\mu\right)\right]}{N-\left(\gamma+1-\mu\right)p}.
	\end{align*}
It finds that  $0\leq \sigma - \zeta < 1$.  For $k_0<h<k$, let us define
	\begin{align*}
	v = \left\{
	\begin{array}{cl}
	k-h&\text{if}\quad u\geq k,\\
	u-h&\text{if}\quad h<u<k,\\
	0&\text{if}\quad u\leq h.
	\end{array}
	\right.
	\end{align*}
	Then $v=0$ on $B_R\backslash A\left(k_0,R\right)$. Applying the Caffarelli-Kohn-Nirenberg inequality and the H\"older inequality we have that
	\begin{align}
	\left(\int_{B_R}\left|x\right|^{-\sigma q}v^qdx\right)^\frac{1}{q} \leq & \int_{B_R}\left|x\right|^{-\zeta}\left|\nabla v\right|dx \notag\\
	=& \int_{\Delta(h,k)}\left|x\right|^{-\frac{\mu p_{\gamma,\mu}^*}{p'}}\left|x\right|^{-\gamma}\left|\nabla v\right|dx\notag\\
	\leq&\left|\Delta(h,k)\right|_{-\mu p_{\gamma,\mu}^*}^{1-\frac{1}{p}}\left(\int_{\Delta(h,k)}\left|x\right|^{-\gamma p}\left|\nabla v\right|^pdx\right)^\frac{1}{p},\label{eq5.9}
	\end{align}
	where $\Delta(h,k) = A\left(h,R\right)\backslash A\left(k,R\right)$ and $q = (\mu p_{\gamma,\mu}^*)/\sigma$ and
	\begin{align*}
		q = \frac{N}{N-\left(\zeta+1- \sigma\right)}=\frac{N}{N-\left(\gamma+1-\mu\right)}>1.
	\end{align*}	
On the other hand, by \eqref{theo_Caccio2} we have
	\begin{align}\label{eq5.10}
		\int_{\Delta(h,k)}\left|x\right|^{-\gamma p}\left|\nabla v\right|^pdx \leq& \frac{C_{\textit{data}}}{R^p}\int_{A(h,2R)}\left|x\right|^{-\gamma p}\left(u-h\right)^pdx + C_{\textit{data}}\left|A\left(h,2R\right)\right|_{-\mu p_{\gamma,\mu}^*}^{1-\delta} \notag\\
		\lesssim& \left(M(2R)-h\right)^pR^{N-\left(\gamma+1\right) p} + R^{\left(N-\mu p_{\gamma,\mu}^*\right)\left(1-\delta\right)}\notag\\
		\lesssim& \,R^{N-\left(\gamma+1\right) p}\left[\left(M(2R)-h\right)^p + R^{\left(N-\mu p_{\gamma,\mu}^*\right)\left(1-\delta\right)-\left(N-\left(\gamma+1\right)p\right)}\right]\notag\\
		=&\,R^{N-\left(\gamma+1\right) p}\left[\left(M(2R)-h\right)^p + R^{\xi_1 p}\right],
	\end{align}
where $\xi_1$ is given by Lemma \ref{thm3.1}. 

\medskip
Further, for $h\leq k_\ell=M(2R) - \frac{1}{2^{\ell+1}}\mathrm{osc}_{B_{2R}}u$, we have
	\begin{align*}
		M(2R) - h\geq M(2R)-k_\ell = \frac{1}{2^{\ell+1}}\mathrm{osc}_{B_{2R}}u \geq R^{\xi_1}.
	\end{align*}
And hence \eqref{eq5.10} yields
	\begin{align}\label{eq5.11}
		\int_{\Delta(h,k)}\left|x\right|^{-\gamma p}\left|\nabla v\right|^pdx \leq & C_{\textit{data}}\, R^{N-\left(\gamma+1\right) p}\left(M(2R)-h\right)^p.
	\end{align}		
Combining \eqref{eq5.9} and \eqref{eq5.11}, for $h\leq k_\ell$ we find that
	\begin{align*}
		\left(k-h\right)\left|A(k,R)\right|_{-\mu p_{\gamma,\mu}^*}^{1-\frac{\gamma+1-\mu}{N}} \leq \left|\Delta(h,k)\right|_{-\mu p_{\gamma,\mu}^*}^{1-\frac{1}{p}}C_{\textit{data}}^{\frac{1}{p}}R^{\frac{N-\left(\gamma+1\right)p}{p}}\left(M(2R)-h\right).
	\end{align*}
Applying the above inequality with $k=k_i=M(2R)-\frac{1}{2^{i+1}}\mathrm{osc}_{B_{2R}}u$	 and $h=k_{i-1}$ we have
	\begin{align*}
		\left|A(k_i,R)\right|_{-\mu p_{\gamma,\mu}^*}^{1-\frac{\gamma+1-\mu}{N}} \leq \left|\Delta(k_{i-1},k_i)\right|_{-\mu p_{\gamma,\mu}^*}^{1-\frac{1}{p}}C_{\textit{data}}^{\frac{1}{p}}R^{\frac{N-\left(\gamma+1\right)p}{p}},\quad i=1,2,...,\ell,
	\end{align*}
which implies
	\begin{align}\label{eq5.12}
		\left|A(k_\ell,R)\right|_{-\mu p_{\gamma,\mu}^*}^{\left(1-\frac{\gamma+1-\mu}{N}\right)\frac{p}{p-1}} \leq \left|\Delta(k_{i-1},k_i)\right|_{-\mu p_{\gamma,\mu}^*}C_{\textit{data}}^{\frac{1}{p-1}}R^{\frac{N-\left(\gamma+1\right)p}{p-1}},\quad i=1,2,...,\ell.
	\end{align}	
Taking a sum over $i$ from $1$ to $\ell$, we obtain
	\begin{align*}
		\ell\left|A(k_\ell,R)\right|_{-\mu p_{\gamma,\mu}^*}^{\left(1-\frac{\gamma+1-\mu}{N}\right)\frac{p}{p-1}} \leq& \left|A(k_0,R)\right|_{-\mu p_{\gamma,\mu}^*}C_{\textit{data}}^{\frac{1}{p-1}}R^{\frac{N-\left(\gamma+1\right)p}{p-1}}.
	\end{align*}
By assumption \eqref{eq5.6}, we imply that
	\begin{align*}
		\left(\frac{\left|A(k_\ell,R)\right|_{-\mu p_{\gamma,\mu}^*}}{R^{N-\mu p_{\gamma,\mu}^*}}\right)^{\frac{N-\left(\gamma+1-\mu\right)}{N}\frac{p}{p-1}} \leq &\frac{\theta_1}{\ell}C_{\textit{data}}^\frac{1}{p-1}.
	\end{align*}	
The proof is complete.	
\end{proof}

\medskip
Similarly as in the previous case we have following lemma:
\begin{lemma}[Case $\gamma\leq \mu<0$]\label{lem_5.5}
	Let 
	\[
	k_0 = \frac{M(2R) + m(2R)}{2}.
	\]
	Assume that the solution $u$ is bounded and satisfies the condition
	\begin{align}\label{eq5.22}
	\left|A\left(k_0,R\right)\right|\leq \theta_2\left|B_R\left(0\right)\right|
	\end{align}
	for some $\theta_2 < 1$. Then we have 
	\begin{align}\label{eq5.24}
	\left(\frac{\left|A\left(k_\ell,R\right)\right|}{R^N}\right)^{\frac{N-\left(\gamma+1\right)}{N}\frac{p}{p-1}}\leq \frac{\theta_2}{\ell}C_{\textit{data}}^{\frac{1}{p-1}},
	\end{align}
	for any integer number $\ell$ satisfying 
	\begin{align*}
	\mathrm{osc}_{B_{2R}}u\geq 2^{\ell+1}R^{\xi_2},
	\end{align*}	
	where $\xi_2$ as in Lemma \ref{lmz02} and
	the level $k_\ell$ are defined by $k_\ell=M(2R)-\frac{1}{2^{\ell+1}}\mathrm{osc}_{B_{2R}}u$.
\end{lemma}
\begin{proof}
	Similarly, for $k_0<h<k$, let us define
	\begin{align*}
	v = \left\{
	\begin{array}{cl}
	k-h&\text{if}\quad u\geq k,\\
	u-h&\text{if}\quad h < u <k,\\
	0&\text{if}\quad u\leq h.
	\end{array}
	\right.
	\end{align*}
	Then $v=0$ on $B_R\backslash A\left(k_0,R\right)$. Applying the Caffarelli-Kohn-Nirenberg inequality (Lemma \ref{lem_CKN2}) with $p=1$, $\beta = \frac{N}{N-\left(\gamma+1\right)}$, $\mu=0$ and $-1<\gamma\leq 0$ we have
	\begin{align*}
	\left(\int_{B_R}v^\frac{N}{N-\left(\gamma+1\right)}dx\right)^\frac{N-\left(\gamma+1\right)}{N} \lesssim& \int_{B_R}\left|x\right|^{-\gamma}\left|\nabla v\right|dx\notag\\
	\leq&\left|\Delta(h,k)\right|^{1-\frac{1}{p}}\left(\int_{B_R}\left|x\right|^{-\gamma p}\left|\nabla v\right|^pdx\right)^\frac{1}{p},
	\end{align*}
	where $\Delta(h,k)=A\left(h,R\right)\backslash A\left(k,R\right)$. 
	
	\medskip
	On the other hand by \ref{theo_Caccio2}, we have 
	\begin{align}\label{eq5.25}
	\int_{B_R}\left|x\right|^{-\gamma p}\left|\nabla v\right|^pdx \leq&C_{\textit{data}}\, R^{N-\left(\gamma+1\right)p}\left(M(2R)-h\right)^p + C_{\textit{data}}\,R^{-\mu p_{\gamma,\mu}^*\left(1-\delta\right)}\left|A\left(h,2R\right)\right|^{1-\delta} \notag\\
	\leq&C_{\textit{data}}\,R^{N-\left(\gamma+1\right)p}\left[\left(M(2R)-h\right)^p +R^{\xi_2 p}\right].
	\end{align}
	For $h\leq k_{\ell} = M(2R)-\frac{1}{2^{\ell+1}}\mathrm{osc}_{B_{2R}}u$ we have
	\begin{align*}
	M(2R) - h\geq M(2R) - k_{\ell} = \frac{1}{2^{\ell+1}}\mathrm{osc}_{B_{2R}}u \geq R^{\zeta_2}
	\end{align*}
	and hence \eqref{eq5.25} yields
	\begin{align*}
	\int_{B_R}\left|x\right|^{-\gamma p}\left|\nabla v\right|^pdx \leq& C_{\textit{data}}\, R^{N-\left(\gamma+1\right)p}\left(M(2R)-h\right)^p.
	\end{align*}	
	Combining all these facts, for $h\leq k_\ell$	we find that
	\begin{align}\label{eq5.26}
	(k-h)\left|A\left(k,R\right)\right|^\frac{N-\left(\gamma+1\right)}{N} \leq \left|\Delta\left(h,k\right)\right|^{1-\frac{1}{p}}C_{\textit{data}}^{\frac{1}{p}} R^{\frac{N-\left(\gamma+1\right)p}{p}}\left(M(2R)-h\right).
	\end{align}
	Applying \eqref{eq5.26} with $k=k_i=M(2R)-\frac{1}{2^{\ell+1}}\mathrm{osc}_{B_{2R}}u$	 and $h=k_{i-1}$ we have
	\begin{align*}
	\left|A\left(k_i,R\right)\right|^\frac{N-\left(\gamma+1\right)}{N} \leq \left|\Delta\left(k_{i-1},k_i\right)\right|^{1-\frac{1}{p}} C_{\textit{data}}^{\frac{1}{p}}R^{\frac{N-\left(\gamma+1\right)p}{p}},\quad i=1,2,...,\ell.
	\end{align*}
	And since $k_i\leq k_\ell$ we arrive at
	\begin{align*}
	\left|A\left(k_\ell,R\right)\right|^{\frac{N-\left(\gamma+1\right)}{N}\frac{p}{p-1}} \leq \left|\Delta\left(k_{i-1},k_i\right)\right| C_{\textit{data}}^{\frac{1}{p-1}}R^{\frac{N-\left(\gamma+1\right)p}{p-1}},\quad i=1,2,...,\ell.
	\end{align*}
	Summing over $i$ from $1$ to $\ell$ and using \eqref{eq5.22} one has
	\begin{align*}
	\ell\left|A\left(k_\ell,R\right)\right|^{\frac{N-\left(\gamma+1\right)}{N}\frac{p}{p-1}} \leq \left|A\left(k_{0},R\right)\right| C_{\textit{data}}^{\frac{1}{p-1}}R^{\frac{N-\left(\gamma+1\right)p}{p-1}}\leq\theta_2C_{\textit{data}}^{\frac{1}{p-1}} R^\frac{\left[N-\left(\gamma+1\right)\right]p}{p-1}.
	\end{align*}
	which implies \eqref{eq5.24}.	
\end{proof}

\subsection{Proof of main result}

\medskip
\qquad With the aid of lemmas in previous subsections we now can prove the local H\"older continuity of solution to \eqref{eq1.1}. 

\vspace{0.2cm}\noindent
{\bf $\blacklozenge$} \,\, First we assume $0 \leq \mu < \gamma + 1$. In this case, it is noted that one of two functions $\pm u$ satisfies the property
\begin{equation}\label{eq6.1}
	|A(k_0, R)|_{-\mu p^*_{\gamma,\mu}} \leq \frac{1}{2} |B_R(0)|_{-\mu p^*_{\gamma,\mu}},
\end{equation}
where $k_0$ is defined by 
\[
	k_0 = \frac{M(2R)+ m(2R)}{2}.
\]
So without loss of generality, we can assume that the solution $u$ satisfies \eqref{eq6.1}. Then we apply Lemma \ref{thm3.1} with $k_0$ is replaced by
\[
	k_\ell = M(2R) - \frac{1}{2^{\ell + 1}}\text{osc}_{B_{2R}}u > k_0
\]
and obtain
\begin{multline*}
	\sup\limits_{B_{R/2}}\left(u-k_\ell\right) \leq R^{\xi_1} + 2^{\frac{1+\alpha_1}{\alpha_1^{2}}}C_{\textit{data}}^{\frac{1}{\alpha_1p_{\gamma,\mu}^*}}\left(\frac{\left|A(k_\ell,R)\right|_{-\mu p_{\gamma,\mu}^*}}{R^{N-\mu p^*_{\gamma,\mu}}}\right)^\frac{\alpha_1}{p_{\gamma,\mu}^*} \\
	\times \left(\frac{1}{R^{N-\mu p_{\gamma,\mu}^*}}\int_{A(k_\ell,R)}\left|x\right|^{-\mu p_{\gamma,\mu}^*}\left(u-k_\ell\right)^{p_{\gamma,\mu}^*}dx\right)^\frac{1}{p_{\gamma,\mu}^*}, 
\end{multline*}
which implies that 
\begin{align}\label{eq5.14}
	\sup\limits_{B_{R/2}}\left(u-k_\ell\right) \leq R^{\xi_1} + 2^{\frac{1+\alpha_1}{\alpha_1^{2}}}C_{\textit{data}}^{\frac{1}{\alpha_1p_{\gamma,\mu}^*}}\left(\frac{\left|A(k_\ell,R)\right|_{-\mu p_{\gamma,\mu}^*}}{R^{N-\mu p^*}}\right)^\frac{\alpha_1}{p_{\gamma,\mu}^*}\sup\limits_{B_{R}}\left(u-k_\ell\right).
\end{align}	
	Let us now choose $\ell$ in such a way that 
		\begin{align*}
			2^{\frac{1+\alpha_1}{\alpha_1^{2}}}C_{\textit{data}}^{\frac{1}{\alpha_1p_{\gamma,\mu}^*}}C_{\textit{data}}^{\frac{1}{p}\frac{N}{N-\left(\gamma+1-\mu\right)}\frac{\alpha_1}{p_{\gamma,\mu}^*}}\left(2\ell\right)^{-\frac{N}{N-\left(\gamma+1-\mu\right)}\frac{p-1}{p}\frac{\alpha_1}{p_{\gamma,\mu}^*}}\leq \frac{1}{2}.
		\end{align*}	
	If $\mathrm{osc}_{B_{2R}}u \geq 2^{\ell+1}R^{\xi_1}$, then by Lemma \ref{lem_5.2}	we deduce from \eqref{eq5.14} that
		\begin{align}\label{eq6.2}
			\sup_{B_{R/2}}u - k_\ell\leq R^{\xi_1} +\frac{1}{2}\left(\sup_{B_{R/2}}u-k_\ell\right).
		\end{align}
	Subtracting	both sides by $\inf_{B_{R/2}}u$ we arrive at
		\begin{align}
			\mathrm{osc}_{B_{R/2}}u \leq \left(1-\frac{1}{2^{\ell+2}}\right)\mathrm{osc}_{B_{2R}}u + R^{\xi_1}.
		\end{align}
	Either case we have $\mathrm{osc}_{B_{2R}}\leq 2^{\ell+1}R^{\xi_1}$. Thus in any case, we find that	
		\begin{align*}
			\mathrm{osc}_{B_{R/2}}u \leq& \left(1-\frac{1}{2^{\ell+2}}\right)\mathrm{osc}_{B_{2R}}u + 2^{\ell+1}R^{\xi_1}\\
			=&\left(\frac{1}{4}\right)^{\lambda}\mathrm{osc}_{B_{2R}}u + 2^{\ell+1} R^{\xi_1}
		\end{align*}
	where $\lambda = -\log_4\left(1-\frac{1}{2^{\ell+2}}\right)$. 
	
	\vspace{0.2cm}
	If necessary we can replace $\xi_1$ in above inequality by $\min\{\xi_1, \lambda/2\}$. Hence we can assume that $\xi_1 < \lambda$. And then by applying Lemma \ref{lem_iteration2} with
	$$
		\tau=\frac{1}{4}, \quad \alpha = -\log_4\left(1-\frac{1}{2^{\ell+2}}\right)
	$$ 
	to get the following estimate for every $r<R$,
		\begin{align*}
			\mathrm{osc}_{B_{r}}u \leq& \left(\frac{r}{R}\right)^{\xi_1}\mathrm{osc}_{B_{R}}u + 2^{\ell+1}r^{\xi_1}.
		\end{align*}	

\vspace{0.2cm}\noindent
{\bf $\blacklozenge$} \,\,  Next the H\"older continuity of solution $u$ in the case $\gamma \leq \mu < 0$ can be proved by the same way as the first case. Here we apply Lemma \ref{lmz02} with $k_\ell$ instead of $k_0$ and get
	\begin{gather*}
	\sup\limits_{B_{R/2}}\left(u-k_\ell\right) \leq  R^{\xi_2} + 2^{\frac{1+\alpha_2}{\alpha_2^{2}}}C_{\textit{data}}^{\frac{1}{\alpha_2p}}\left(\frac{\left|A(k_\ell,R)\right|}{R^{N}}\right)^\frac{\alpha_2}{p}\left(\frac{1}{R^{N}}\int_{A(k_\ell,R)}\left(u-k_\ell\right)^{p}dx\right)^\frac{1}{p},
	\end{gather*}
	which implies
	\begin{align}\label{eq5.27}
	\sup\limits_{B_{R/2}}\left(u-k_\ell\right) \leq R^{\xi_2} + 2^{\frac{1+\alpha_2}{\alpha_2^{2}}}C_{\textit{data}}^{\frac{1}{\alpha_2p}}\left(\frac{\left|A(k_\ell,R)\right|}{R^{N}}\right)^\frac{\alpha_2}{p}\sup\limits_{B_{R}}\left(u-k_\ell\right).
	\end{align}	
	Let us now choose $\ell$ in such a way that 
	\begin{align*}
	2^{\frac{1+\alpha_2}{\alpha_2^{2}}}C_{\textit{data}}^{\frac{1}{\alpha_2p}}C_{\textit{data}}^{\frac{N}{N-\left(\gamma+1\right)}\frac{\alpha_2}{p^2}}(2\ell)^{-\frac{N}{N-\left(\gamma+1\right)}\frac{\left(p-1\right)\alpha_2}{p^2}}\leq \frac{1}{2}.
	\end{align*}	
	If $\mathrm{osc}_{B_{2R}}u\geq 2^{\ell+1}R^{\xi_2}$, then by Lemma \ref{lem_5.5}	we deduce from \eqref{eq5.27} that
	\begin{align}
		\sup_{B_{R/2}}u  - k_\ell\leq R^{\xi_2} +\frac{1}{2}\Big(\sup_{B_{R/2}}u -k_\ell\Big).
	\end{align}
	Subtracting	both sides by $\inf_{B_{R/2}}u $ we arrive at
	\begin{align}
	\mathrm{osc}_{B_{R/2}}u \leq \left(1-\frac{1}{2^{\ell+2}}\right)\mathrm{osc}_{B_{2R}}u + R^{\xi_2}.
	\end{align}
	Either case we have $\mathrm{osc}_{B_{2R}}\leq 2^{\ell+1}R^{\xi_2}$. Thus in any case, we find that	
	\begin{align*}
	\mathrm{osc}_{B_{R/2}}u \leq& \left(1-\frac{1}{2^{\ell+2}}\right)\mathrm{osc}_{B_{2R}}u + 2^{\ell+1} R^{\xi_2}\\
	=&\left(\frac{1}{4}\right)^{\lambda}\mathrm{osc}_{B_{2R}}u + 2^{\ell+1} R^{\xi_2}
	\end{align*}
	where $\lambda = - \log_4\left(1-\frac{1}{2^{\ell+2}}\right)$. Assuming $\xi_2 < \lambda$ and applying Lemma \ref{lem_iteration2} we obtain
	\begin{align*}
	\mathrm{osc}_{B_{r}}u \leq& \left(\frac{r}{R}\right)^{\xi_2}\mathrm{osc}_{B_{R}}u + 2^{\ell+1}r^{\xi_2},
	\end{align*}
	for every $r<R\leq R_0$. The proof of our theorem is completed.

\medskip
\medskip

\noindent
{\bf Acknowledgement:} The final work of this paper was done when the authors visited Vietnam Institute for Advanced Study in Mathematics (VIASM). We would like to thank VIASM for their supports.

\vspace{0.2cm}
The first author is supported by Ho Chi Minh City University of Technology and Education, Vietnam. The second and third authors are supported by University of Economics Ho Chi Minh City, Vietnam.

\end{document}